\documentclass{eptcs-modified}
 \providecommand{\event}{} 
\usepackage{amsmath}
\usepackage{amssymb}
\usepackage{amsthm}
\usepackage{amscd}
\usepackage{graphics}
\usepackage{soul}

\title{Sequential discontinuity and first-order problems}
\author{
Arno Pauly
\institute{Swansea University\\Swansea, UK}
\email{Arno.M.Pauly@gmail.com}
\and
Giovanni Sold\`a
\institute{Ghent University\\Ghent, Belgium}
\email{Giovanni.A.Solda@gmail.com}
}
\def\titlerunning{Sequential discontinuity and first-order problems}
\def\authorrunning{A. Pauly \& G.~Sold\`a}

\begin{document}
\theoremstyle{definition}
\newtheorem{theorem}{Theorem}
\newtheorem{definition}[theorem]{Definition}
\newtheorem{problem}[theorem]{Problem}
\newtheorem{assumption}[theorem]{Assumption}
\newtheorem{corollary}[theorem]{Corollary}
\newtheorem{proposition}[theorem]{Proposition}
\newtheorem{lemma}[theorem]{Lemma}
\newtheorem{observation}[theorem]{Observation}
\newtheorem{fact}[theorem]{Fact}
\newtheorem{question}[theorem]{Open Question}
\newtheorem{conjecture}[theorem]{Conjecture}
\newtheorem{example}[theorem]{Example}
\newtheorem{remark}[theorem]{Remark}
\newcommand{\dom}{\operatorname{dom}}
\newcommand{\id}{\textnormal{id}}
\newcommand{\Cantor}{{\{0, 1\}^\mathbb{N}}}
\newcommand{\Baire}{{\mathbb{N}^\mathbb{N}}}
\newcommand{\Lev}{\textnormal{Lev}}
\newcommand{\hide}[1]{}
\newcommand{\mto}{\rightrightarrows}
\newcommand{\uint}{{[0, 1]}}
\newcommand{\bft}{\mathrm{BFT}}
\newcommand{\lbft}{\textnormal{Linear-}\mathrm{BFT}}
\newcommand{\pbft}{\textnormal{Poly-}\mathrm{BFT}}
\newcommand{\sbft}{\textnormal{Smooth-}\mathrm{BFT}}
\newcommand{\ivt}{\mathrm{IVT}}
\newcommand{\cc}{\textrm{CC}}
\newcommand{\lpo}{\textrm{LPO}}
\newcommand{\llpo}{\textrm{LLPO}}
\newcommand{\aou}{AoU}
\newcommand{\Ctwo}{C_{\{0, 1\}}}
\newcommand{\name}[1]{\textsc{#1}}
\newcommand{\C}{\textrm{C}}
\newcommand{\CC}{\textrm{CC}}
\newcommand{\UC}{\textrm{UC}}
\newcommand{\ic}[1]{\textrm{C}_{\sharp #1}}
\newcommand{\xc}[1]{\textrm{XC}_{#1}}
\newcommand{\me}{\name{P}.~}
\newcommand{\etal}{et al.~}
\newcommand{\eval}{\operatorname{eval}}
\newcommand{\rank}{\operatorname{rank}}
\newcommand{\Sierp}{Sierpi\'nski }
\newcommand{\isempty}{\operatorname{IsEmpty}}
\newcommand{\spec}{\textrm{Spec}}
\newcommand{\leqW}{\leq_{\textrm{W}}}
\newcommand{\leW}{<_{\textrm{W}}}
\newcommand{\equivW}{\equiv_{\textrm{W}}}
\newcommand{\geqW}{\geq_{\textrm{W}}}
\newcommand{\pipeW}{|_{\textrm{W}}}
\newcommand{\nleqW}{\nleq_{\textrm{W}}}
\newcommand{\leqsW}{\leq_{\textrm{sW}}}
\newcommand{\lesW}{<_{\textrm{sW}}}
\newcommand{\equivsW}{\equiv_{\textrm{sW}}}
\newcommand{\geqsW}{\geq_{\textrm{sW}}}
\newcommand{\pipesW}{|_{\textrm{sW}}}
\newcommand{\nleqsW}{\nleq_{\textrm{sW}}}
\newcommand{\Det}{\textrm{Det}}

\newcommand{\imp}{\rightarrow}
\newcommand{\Imp}{\Rightarrow}
\newcommand{\omone}{\ensuremath{\omega + 1}}
\newcommand{\Nb}{\mathbb{N}}
\newcommand{\accn}{\mathsf{ACC}_{\Nb}}
\newcommand{\seqaccn}{\mathsf{SEQACC}_{\Nb}}
\newcommand{\la}{\langle}
\newcommand{\ra}{\rangle}
\newcommand{\conc}{^{\smallfrown}}
\newcommand{\neqX}[1]{\mathalpha{\neq}_{#1}}
\newcommand{\cbr}{\mathrm{CBr}}

\newcounter{saveenumi}
\newcommand{\seti}{\setcounter{saveenumi}{\value{enumi}}}
\newcommand{\conti}{\setcounter{enumi}{\value{saveenumi}}}

\maketitle

\begin{abstract}
    We explore the low levels of the structure of the continuous Weihrauch degrees of first-order problems. In particular, we show that there exists a minimal discontinuous first-order degree, namely that of $\accn$, without any determinacy assumptions. The same degree is also revealed as the least sequentially discontinuous one, i.e.~the least degree with a representative whose restriction to some sequence converging to a limit point is still discontinuous.

    The study of games related to continuous Weihrauch reducibility constitutes an important ingredient in the proof of the main theorem. We present some initial additional results about the degrees of first-order problems that can be obtained using this approach. 
\end{abstract}


\section{Introduction}
A well-known continuity notion is that convergent sequences get mapped to convergent sequences. As spaces of relevance for computable analysis are sequential, this notion adequately characterizes continuity for functions for us. Multivalued functions, on the other hand, can exhibit a very different behaviour. To one extreme, the multivalued function $\mathrm{NON} : \Cantor \mto \Cantor$ with $q \in \mathrm{NON}(p)$ iff $q \nleq_\mathrm{T} p$ is discontinuous, yet whenever restricted to a subset of its domain of cardinality less than the continuum even has a constant realizer. It thus seems a natural question to ask when discontinuity of a multivalued function stems from its behaviour restricted to some convergent sequence.

The answer to this turns out to be related to the notion of first-order part of a Weihrauch degree, proposed by Dzafahrov, Solomon and Yokoyama \cite{solomon}. It is defined as: $$^1(f) := \max_{\leqW} \{g : \subseteq \Baire \mto \mathbb{N} \mid g \leqW f\}$$

It turns out that there is a weakest Weihrauch degree with non-trivial first-order part (up to continuous reductions), namely the degree of $\accn$ which is the restriction of $\C_\mathbb{N}$ to sets $A \subseteq \mathbb{N}$ with $|\mathbb{N} \setminus A| \leq 1$ studied in \cite{kreuzer}. We show the following:

\begin{theorem}
\label{theo:main}
The following are equivalent for a multi-valued function $f : \subseteq \Baire \mto \Baire$:
\begin{enumerate}
\item $^1(f)$ is discontinuous.
\item There exists a convergent sequence $(a_n)_{n \in \mathbb{N}}$ with $\overline{\{a_n \mid n \in \mathbb{N}\}} \subseteq \dom(f)$ such that $f|_{\overline{\{a_n \mid n \in \mathbb{N}\}}}$ is discontinuous.
\item $\accn \leqW^* f$
\end{enumerate}
\end{theorem}

We prove our main theorem in two separate parts. The equivalence between $2.$ and $3.$ follows by a direct analysis of how a multivalued function with domain $\omone$ can be discontinuous; it is given as Lemma \ref{sec:equiv-disc-accn} in Section \ref{sec:domainanalysis}. The equivalence between $1.$ and $3.$ is built on the generalization of Wadge-style games to multivalued functions proposed in \cite{pauly-nobrega}. We elucidate this in Section \ref{sec:games}. The relevant half of Theorem \ref{theo:main} is summarized in Corollary \ref{corr:mainhalf}, but we obtain additional results of independent interest.

\section{Background}
We assume that the reader is familiar with the basics of represented spaces and Weihrauch reducibility, and refer to \cite{pauly-synthetic} as reference for the former and \cite{pauly-handbook} as reference for the latter. In this section we provide the lesser known definitions we require.

\subsection{Represented spaces}

\subsubsection*{Completions of a represented space}
We use the notions of completions of represented spaces and multiretraceability discussed in \cite{damir,gherardi2,gherardi6}, but we will introduce these concepts somewhat differently than those works. The reason is that \emph{the completion} of a represented space as defined there is not compatible with computable isomorphisms, and we thus prefer to speak of \emph{a completion} -- in particular since for our purposes any completion in the sense we introduce below is equally good.

\begin{definition}[{cf.~\cite[Proposition 2.6]{gherardi6}}]
A represented space $\mathbf{X}$ is \emph{multiretraceable}, if every computable partial function $f : \subseteq \Cantor \to \mathbf{X}$ has a total computable extension $F : \Cantor \to \mathbf{X}$.
\end{definition}

\begin{definition}\label{def:completion}
A \emph{completion} of a represented space $\mathbf{X}$ consists of a computable embedding $\iota : \mathbf{X} \hookrightarrow \mathbf{Y}$ into a multitraceable space $\mathbf{Y}$.
\end{definition}

\begin{proposition}[\cite{damir}, cf.~\cite{gherardi6}]
Every represented space has a completion.
\end{proposition}

\begin{proposition}
\label{prop:completions}
If $\iota_1 : \mathbf{X} \hookrightarrow \mathbf{Y}_1$ and $\iota_2 : \mathbf{X} \hookrightarrow \mathbf{Y}_2$ are completions of $\mathbf{X}$, then there is a computable multivalued function $T : \mathbf{Y}_1 \mto \mathbf{Y}_2$ such that $\iota_2 = T \circ \iota_1$.
\begin{proof}
Let $\delta_1 : \subseteq \Cantor \to \mathbf{Y}_1$ be the representation of $\mathbf{Y}_1$. We note that $\iota_2 \circ \iota_1^{-1} \circ \delta_1 : \subseteq \Cantor \to \mathbf{Y}_2$ is a partial computable function, and thus has computable total extension $F : \Cantor \to \mathbf{Y}_2$. Now $T := F \circ \delta_1^{-1}$ is the desired map.
\end{proof}
\end{proposition}

We will use $\overline{\mathbf{X}}$ to denote some completion of $\mathbf{X}$ in the following, and will prove for each use that the specific choice of completion is irrelevant by invoking Proposition \ref{prop:completions}.

\subsubsection*{Layering of spaces}
\begin{definition}
Given two represented spaces $\mathbf{X}$, $\mathbf{Y}$ we define the represented space $\frac{\mathbf{X}}{\mathbf{Y}}$ to have the underlying set $\{\frac{x}{\cdot} \mid x \in \mathbf{X}\} \cup \{ \frac{\cdot}{y} \mid y \in \mathbf{Y}\}$ with the representation defined as follows: Let $\delta_\mathbf{X} : \subseteq \Cantor \to \mathbf{X}$, $\delta_\mathbf{Y} : \subseteq \Cantor \to \mathbf{Y}$ be the original representations. Then we introduce $\delta_{\mathrm{new}}(q) = \frac{\cdot}{\delta_\mathbf{Y}(q)}$ for any $q \in \dom(\delta_\mathbf{Y}) \subseteq \Cantor$ and $\delta_{\mathrm{new}}(w2p) = \frac{\delta_\mathbf{X}(p)}{\cdot}$ for $w \in \{0,1\}^*$, $p \in \dom(\delta_\mathbf{X})$.
\end{definition}
  
 Informally spoken, a point in $\frac{\mathbf{X}}{\mathbf{Y}}$ starts off looking like a point in $\mathbf{Y}$, and it might indeed be a point in $\mathbf{Y}$. It can also, at any moment, change to be some arbitrary point in $\mathbf{X}$ instead. For example, we find that $\mathbb{S} \cong \frac{\mathbf{1}}{\mathbf{1}}$.
 
 For subsets $A \subseteq \mathbf{X}$, $B \subseteq \mathbf{Y}$ we shall write $\frac{A}{\cdot} = \{\frac{x}{\cdot} \mid x \in A\}$ and $\frac{\cdot}{B} = \{\frac{\cdot}{y} \mid y \in B\}$.
 \begin{observation}
 We find that $\mathcal{O}(\frac{\mathbf{X}}{\mathbf{Y}}) = \{\frac{U}{\cdot} \mid U \in \mathcal{O}(\mathbf{X})\} \cup \{\frac{X}{\cdot} \cup \frac{\cdot}{V} \mid V \in \mathcal{O}(\mathbf{Y})\}$.
 \end{observation}
 
\begin{proposition}
The space $\frac{\mathbf{X}}{\mathbf{Y}}$ is scattered iff both $\mathbf{X}$ and $\mathbf{Y}$ are; and then furthermore: \[\cbr(\frac{\mathbf{X}}{\mathbf{Y}}) = \cbr(\mathbf{X}) + \cbr(\mathbf{Y})\]
\end{proposition}
\begin{proof}
We find that $\frac{x}{\cdot} \in \frac{\mathbf{X}}{\mathbf{Y}}$ is isolated iff $x \in \mathbf{X}$ is isolated. On the other hand, $\frac{\cdot}{y} \in \frac{\mathbf{X}}{\mathbf{Y}}$ is isolated only if $\mathbf{X}$ is the empty space and $y \in \mathbf{Y}$ is isolated. If $\mathbf{Z}'$ denotes the Cantor-Bendixson derivative, it follows that for $\mathbf{X} \neq \emptyset$, we have $\left (\frac{\mathbf{X}}{\mathbf{Y}} \right )' = \frac{\mathbf{X}'}{\mathbf{Y}}$. Together with the observation that $\frac{\emptyset}{\mathbf{Y}} \cong \mathbf{Y}$, the claim follows.
\end{proof}

\subsubsection*{Some represented spaces of particular interest}

\begin{definition}
The represented space $(\omone)$ has the underlying set $\mathbb{N} \cup \{\omega\}$ and comes with the representation $\delta_{\omone} : \Cantor \to (\omone)$ where $\delta_{\omone}(0^n1p) = n$ for any $p \in \Cantor$ and $\delta_{\omone}(0^\omega) = \omega$.
\end{definition}
For added clarity, we note that the represented space $\omone$ as given above is \emph{not} multiretraceable, and hence the obvious embedding $\mathbb{N}\to\omone$ is not a completion in the sense of Definition \ref{def:completion}: indeed, consider for instance the partial map $f:\subseteq 2^{\mathbb{N}}\to \omone$ such that on input $p$ it outputs $0$ if the first $1$ in $p$ occurs in an even position and $1$ if it occurs in an odd position. 

\subsection{Weihrauch reducibility}
We use $f \leqW g$ to denote that $f$ is Weihrauch reducible to $g$. By $f \leqW^* g$ we denote that $f$ is continuous Weihrauch reducible to $g$, meaning that the reduction witness are merely require to be continuous rather than computable functions. We can also view $f \leqW^* g$ to mean that $f$ is Weihrauch reducible to $g$ relative to some oracle, as continuity is just computability relative to an arbitrary oracle.

\subsubsection*{The first-order part of a Weihrauch degree}
As mentioned in the introduction, the first-order part of a Weihrauch degree $f$, denoted by $^1f$, was proposed by Dzafahrov, Solomon and Yokoyama \cite{solomon}. It is defined as: $$^1(f) := \max_{\leqW} \{g : \subseteq \Baire \mto \mathbb{N} \mid g \leqW f\}$$
As the Weihrauch lattice is not complete \cite{paulykojiro}, an explicit construction is needed to show that the first-order part indeed exists. This is straight-forward, however. The first-order part was thoroughly studied in \cite{soldavalenti}, and has turned out to be a convenient tool to obtain separations between Weihrauch degrees (e.g.~\cite{pauly-valenti,damir-reed-manlio,paulycipriani1}).

\section{Discontinuity of multivalued functions on $\omone$}
\label{sec:domainanalysis}

One component of the our main theorem is provided by an analysis of the possible Weihrauch degrees for multivalued functions with domain $\omone$. We start by providing a multivalued function with domain $\omone$ which is equivalent to $\accn$.

\begin{definition}
  $\seqaccn: \omone \mto \Nb$ is the problem defined as $\seqaccn(\la n,m\ra)= \Nb\setminus\{n\}$, and $\seqaccn(\omega)=\Nb$.
\end{definition}

\begin{lemma}\label{sec:accn-seqaccn}
  $\accn\equivsW\seqaccn$
\end{lemma}
\begin{proof}
  We start showing that $\accn\leqsW\seqaccn$. An instance of $\accn$ is an enumeration $g$ of an open set $U$ of $\Nb$ that contains at most one point. We define the inner functional $H$ as follows: at every step $s$, if nothing was enumerated in $U$, we let $H(g)_{\leq s}=0^s$. If instead at stage $s$ $g$ enumerated the number $n$ in $U$, we set $H(g)={0^{\la n,s\ra}}\conc 1^\Nb$. After applying $\seqaccn$, it is clear that we can take the identity as outer functional.

  Next we show that $\seqaccn\leqsW\accn$. Again, we only need to define the inner functional $H$, and we will do it as follows: given a certain sequence $p\in 2^\Nb$, we enumerate nothing in the open set $U$ unless, at step $s$, we see that $p={0^s}\conc 1\conc q$, in which case we enumerate $\pi_1(s)$ in $U$.
\end{proof}

\begin{lemma}\label{sec:equiv-disc-accn}
  The following are equivalent for a partial multivalued function $f:\subseteq \Baire \mto \Baire$:
  \begin{enumerate}
  \item there exists a converging sequence $( a_n)_{n\in\omega}\to a$ such that $a_n\in \dom f$ for every $n\in\omega$, $a\in \dom f$, and $f|_{\overline{\{a_n:n\in\Nb\}}}$ is discontinuous.
  \item $\accn\leqW^* f$
  \end{enumerate}
\end{lemma}
\begin{proof}
  We start by proving $2\imp 1$: by Lemma \ref{sec:accn-seqaccn}, we can suppose that $\seqaccn\leqW^* f$. The image of $\omone$ under the reduction yields the convergent sequence $( a_n)_{n\in\omega}\to a$; and the corresponding restriction $f|_{\overline{\{a_n:n\in\Nb\}}}$ still satisfies $\seqaccn\leqW^* f|_{\overline{\{a_n:n\in\Nb\}}}$, and thus is discontinuous.
  
  Now on to $1\imp 2$: 
  Let $( a_n)_{n\in\omega}\to a$ be a sequence as in the hypotheses. We claim that the following holds:
  \[
    \forall p\in f(a) \ \exists k\in\Nb \ \forall N\in\Nb\ \exists j>N (f(a_j)\cap [p_{\leq k}]=\emptyset).
  \]
  Suppose for a contradiction that this is false. Then, there is $p\in f(a)$ such that for every integer $k$ there is an $N_k$ with the property that for every $a_j$ with $j>N_k$, $f(a_j)\cap [p_{\leq k}]$ is not empty, say that it contains a certain $q_{k,j}$. Then, we can define a realizer $F$ for $f$ that is continuous on $\overline{\{a_n:n\in\Nb\}}$ as follows: for every $i\in\Nb$ let $M_i=\sum_{j<i} N_j+1$, and for every $j\in [M_i,M_{i+1}-1]$ let $F(a_j)=q_{i,j}$ (notice that the $[M_i,M_{i+1}-1]$ partition $\Nb$, hence $F$ is defined on every $a_j$). Moreover, let $F(a)=p$, and let $F(x)\in f(x)$ for every other element $x\in\dom f$. It is immediately verified that this $F$ is a realizer as we wanted it.
  
  Notice that the $p_{\leq k}$ as above form a countable set, and assume we have an enumeration of it, say $( w_i:i\in\Nb ($ (an element might be enumerated several times). Then, we can define the function $\lambda:\Nb\times\Nb\to \Nb$ as $\lambda(n,i)=j$, where $j$ is minimal such that $d(a_j,a)<1/2^n$ and $f(a_j)\cap w_i=\emptyset$ (it is easy to see that such a $j$ exists). We will use the function $\lambda$ and the set $\{w_i:i\in\Nb\}$ as oracles for the functionals witnessing that $\accn\leqW^* f$.

  Let $g\in\dom \accn$ be an enumeration of an open set $U$ that contains at most one point. We define the inner functional $H$ in steps as follows: at step $n$, we check whether a number $i$ is enumerated in $U$ by $g$. If this is the case, the construction stops and we output $H(g)=a_{\lambda(n,i)}$, whereas if this is not the case we set $H(g)(n)=a(n)$. Our assumption that $a_j\to a$ guarantees that the construction works.

  Finally, we define the outer functional $K$ on input $(q,g)$ (where $q\in f(H(g))$) as follows: for every $i\in\Nb$, we check simultaneously whether $q$ extends $w_i$ and whether $i$ is enumerated in $U$ by $g$. Notice that at least one of these searches gives a positive result, since if no $i$ is enumerated in $U$ by $g$, then $H(g)=a$. Suppose then that the search for a certain $i$ gave positive result: then, if the search was successful because we found that $i$ is enumerated in $U$ by $g$, we set $K(q,g)=i+1$, which is then clearly in $\accn(g)$. If instead the search was successful because $q$ extends $w_i$, we set $K(q,g)=i$: by the fact that $q$ extends $w_i$, we know that at no point, while determining $H(g)$, it was found that $i$ is enumerated in $U$, and hence $i\in\accn(g)$.
\end{proof}

\section{Games for Weihrauch reducibility}
\label{sec:games}

In this section we identify the least degree of discontinuity a multivalued function with a given codomain can attain. We begin by introducing a suitable representative of this degree.

\begin{definition}
Let $\neqX{\mathbf{X}} : \overline{\mathbf{X}} \mto \mathbf{X}$ by defined by $y \in \neqX{\mathbf{X}}(x)$ iff $x \neq y$.
\end{definition}

\begin{proposition}
\label{prop:eq-dom}
Let $\iota_1 : \mathbf{X} \hookrightarrow \mathbf{Y}_1$, $\iota_2 : \mathbf{X} \hookrightarrow \mathbf{Y}_2$ be two completions of $\mathbf{X}$, and let $\neqX{\mathbf{X}}^1$ and $\neqX{\mathbf{X}}^2$ be the maps defined using those. Then $\neqX{\mathbf{X}}^1 \equivsW \neqX{\mathbf{X}}^2$.
\begin{proof}
Consider the computable multivalued function $T : \mathbf{Y}_1 \mto \mathbf{Y}_2$ provided by Proposition \ref{prop:completions}. Then $\neqX{\mathbf{X}}^2 \circ T$ tightens $\neqX{\mathbf{X}}^1$, thus $\neqX{\mathbf{X}}^1 \leqW \neqX{\mathbf{X}}^2$. By symmetry, the claim follows.
\end{proof}
\end{proposition}

In the case that $\mathbf{X}$ is empty or consists of a single point the above definition does not entirely make sense. It will be convenient to assume that $\neqX{\emptyset} \equivW 0$, where $0$ is the no-where defined function, and that for every singleton space $\mathbf{1}$ it holds that $\neqX{\mathbf{1}} \equivW \infty$ where $\infty$ is the (artificial) top element in the Weihrauch degrees (see the discussion in \cite[Section 2.1]{paulybrattka4}).

For discrete spaces $\mathbf{X}$, we find that the degree of $\neqX{\mathbf{X}}$ is a familiar principle:

\begin{observation}
\label{obs:neqn}
$\neqX{\Nb} \equivsW \accn$ and $\neqX{\mathbf{n}} \equivsW \mathsf{ACC}_{\mathbf{n}}$.
\end{observation}

We obtain our main result for this section by analysing a generalization of Wadge games for multivalued function between represented spaces. We are primarily following \cite{pauly-nobrega} for this, but similar ideas are present e.g.~in \cite{pequignot, discontinuityproblem}.

\begin{definition}
The Wadge game for a multivalued function $f : \mathbf{X} \mto \mathbf{Y}$ is played by two players, $\mathrm{I}$ and $\mathrm{II}$ taking turns. Player $\mathrm{I}$ always plays a natural number, Player $\mathrm{II}$ can play a natural number or skip her turn. The game is zero-sum, and winning is determined by the following rules, with earlier rules taking precedence over later ones.
\begin{enumerate}
\item If Player $\mathrm{I}$'s moves do not form a name for an element $x \in \mathbf{X}$, Player $\mathrm{I}$ loses the game.
\item If Player $\mathrm{II}$ does not play a number infinitely often, or if the numbers she plays does not form a name for some $y \in \mathbf{Y}$, Player $\mathrm{II}$ loses.
\item Otherwise, Player $\mathrm{II}$ wins iff $y \in f(x)$.
\end{enumerate}
\end{definition}

The motivation behind the definition of the Wadge game is the following rather straight-forward observation:

\begin{observation}
\label{obs:player2wins}
Player $\mathrm{II}$ has a winning strategy for the Wadge game for $f : \mathbf{X} \mto \mathbf{Y}$ iff $f$ is continuous.
\end{observation}

If winning strategies for Player $\mathrm{II}$ witness continuity of the multivalued function $f$, it is very natural to ask what a winning strategy for Player $\mathrm{I}$ would do. If the game is determined, such an analysis gives rise to a dichotomy as follows:

\begin{theorem}[$\mathrm{ZF} + \mathrm{DC} + \mathrm{AD}$]
\label{theo:neqminimal}
A multivalued function $f : \mathbf{Z} \mto \mathbf{X}$ is either continuous or satisfies $\neqX{\mathbf{X}} \leqW^* f$.
\end{theorem}

\begin{proof}
By assumption, the Wadge game for $f : \mathbf{Z} \mto \mathbf{X}$ is determined. By Observation \ref{obs:player2wins}, Player $\mathrm{II}$ wins the game iff $f$ is continuous. Thus, Player $\mathrm{I}$ wins the game iff $f$ is discontinuous.

A winning strategy for Player $\mathrm{I}$ can be described as a function $\lambda : (\mathbb{N} \cup \{\operatorname{skip}\})^* \to \mathbb{N}$ which needs to satisfy that for any sequence $q \in (\mathbb{N} \cup \{\operatorname{skip}\})^\omega$ the sequence \[g(q) := \lambda(q_{<0})\lambda(q_{<1})\lambda(q_{<2})\lambda(q_{<3})\ldots\] constitutes a name for some $x \in \mathbf{Z}$. Moreover, if removing all $\operatorname{skip}$'s from $q$ produces a name for some $y \in \mathbf{X}$, then it must hold that $y \notin f(x)$.

We can obtain a completion $\overline{\mathbf{X}}$ of $\mathbf{X}$ by using $q \in (\mathbb{N} \cup \{\operatorname{skip}\})^\omega$ as names (with some renumbering), and letting any sequence $q$ were removing the $\operatorname{skip}$'s yields a name for some $y \in \mathbf{X}$ be a name for $y \in \overline{\mathbf{X}}$, and letting any other sequence be a name for a fresh element $\bot$. We find that the function $g$ obtained from the winning strategy of Player $\mathrm{I}$ is a realizer of a continuous multivalued function $G : \overline{\mathbf{X}} \mto \mathbf{Z}$. Moreover, we have that $f \circ G : \overline{\mathbf{X}} \mto \mathbf{X}$ tightens $\neqX{\mathbf{X}} :  \overline{\mathbf{X}} \mto \mathbf{X}$, and thus witnesses that $\neqX{\mathbf{X}} \leqW^* f$.
\end{proof}

Theorem \ref{theo:neqminimal} generalizes the approach used by Brattka in \cite{discontinuityproblem} to identify the ``least discontinuous'' problem, assuming enough determinacy. The discontinuity problem $\operatorname{DIS}$ studied by Brattka is essentially the same as $\neqX{\Baire}$ (the problems have formally different domains, but the same realizers).

If we are considering multivalued functions with codomain $\mathbb{N}$ specifically, we do not need to appeal to the axiom of determinacy to render the corresponding games determined. Instead, we prove the following:

\begin{lemma}\label{sec:det-games-N}
    Let $f:\subseteq \Baire \mto \Nb$ be a partial multivalued function. Then, the Wadge game for $f$ is determined.
  \end{lemma}
  \begin{proof}
    We show that if Player I does not have a strategy for the game, then Player II does. To do this, we spell out what it means for Player I to have a winning strategy: Player I has a winning strategy if and only if there is an $x\in \dom f$ such that there are infinitely many $t_0\sqsubset t_1\sqsubset t_2\sqsubset\dots\sqsubset x$ such that, for every $i,j\in\Nb$, there is an $x_j^i\in \dom f$ such that $t_j\sqsubseteq x^i_j$ and $i\not\in f(x^i_j)$, as we now show.

    Suppose that the condition above holds, then the strategy for Player I consists simply in playing $t_j$ as long as Player II responds by playing $\emptyset$: if Player II ever commits to a certain value $j$, then by our assumption we have an available move for player I that proves that Player II's commitment is wrong. If Player II never commits, Player I wins since $x\in \dom f$.

    On the other hand, if Player I has a strategy, we can build an $x$ as above by supposing that Player II constantly skips her turn: it is easy to verify that this describes an $x$ as we want.

    Hence, suppose that no $x$ as above exists: this means that for every $x\in\dom f$, there is $t_x\sqsubset x$ and an integer $i_{t_x}$ such that every for every $y\in \dom f\cap [t_x]$ $i_{t_x}\in f(y)$. This gives rise to a strategy for player II: player II plays $\emptyset$ until player I plays some string in a $t_x$, for some $x$ (if player I never does this, then it is not producing an element of $\dom f$, which guarantees that player II wins), and then responds $i_{t_x}$. 
  \end{proof}

\begin{corollary}
\label{corr:mainhalf}
The following are equivalent for a multivalued function $f : \mathbf{Z} \mto \mathbb{N}$:
\begin{enumerate}
\item $f$ is discontinuous.
\item $\accn \leqW^* f$
\end{enumerate}
\end{corollary}
\begin{proof}
This follows from Theorem \ref{theo:neqminimal}, with Lemma \ref{sec:det-games-N} removing the need to invoke $\mathrm{AD}$ and Observation \ref{obs:neqn} letting us replace $\neqX{\Nb}$ by $\accn$.
\end{proof}

In the very same way we also obtain the following:

\begin{corollary}
The following are equivalent for a multivalued function $f : \mathbf{Z} \mto \mathbf{n}$:
\begin{enumerate}
\item $f$ is discontinuous.
\item $\mathsf{ACC}_\mathbf{n} \leqW^* f$
\end{enumerate}
\end{corollary}

Similar to the first-order part, the $k$-finitary part of a Weihrauch degree $g$ is the maximal degree of some $f : \mathbf{X} \mto \mathbf{k}$ with $f \leqW g$ \cite{paulycipriani1}. We then see that, up to some oracle, the weakest non-trivial $k$-finitary part is just $\mathsf{ACC}_\mathbf{n}$.  

\subsection{On the complexity of $\neqX{\mathbf{X}}$}
\label{subsec:neqx}
Since $\neqX{\mathbf{X}}$ characterizes the least discontinuous degree of discontinuity possible for a multivalued function with codomain $\mathbf{X}$, we will briefly explore how the degree of $\neqX{\mathbf{X}}$ varies with the space $\mathbf{X}$. To roughly summarize our results, sufficiently complicated spaces all share the same degree (that of $\neqX{\Baire}$), while for scattered spaces we see substantial variation.

\begin{proposition}
\label{prop:neqxsubspace}
Let $\emptyset \neq \mathbf{X} \subseteq \mathbf{Y}$. Then $\neqX{\mathbf{Y}} \leqsW \neqX{\mathbf{X}}$.
\begin{proof}
By the defining properties of the completion, the computable partial map $\id : \subseteq \mathbf{Y} \to \mathbf{X}$ extends to a computable multivalued map $K : \overline{\mathbf{Y}} \mto \overline{\mathbf{X}}$ satisfying $K(x) = x$ for all $x \in \mathbf{X}$. This serves as inner reduction witness, together with $\id : \mathbf{X} \to \mathbf{Y}$ as outer reduction witness.
\end{proof}
\end{proposition}

\begin{proposition}
\label{prop:neqxsurjection}
Let $s : \mathbf{X} \to \mathbf{Y}$ be a computable surjection, and let $t : \mathbf{Y} \mto \mathbf{X}$ be computable with $s \circ t = \id_\mathbf{Y}$. Then $\neqX{\mathbf{X}} \leqsW \neqX{\mathbf{Y}}$.
\begin{proof}
The reduction is witnessed by the lifting of $s$ to $\overline{s} : \overline{\mathbf{X}} \mto \overline{\mathbf{Y}}$ together with $t$.
\end{proof}
\end{proposition}

\begin{corollary}
Let $\mathbf{X}$ be an uncountable Polish space. Then $\neqX{\mathbf{X}} \equivW^* \neqX{\Baire}$.
\begin{proof}
An uncountable Polish space contains a perfect subset, i.e.~a subset homeomorphic to Cantor space\footnote{This claim drastically fails in the computable setting. Gregoriades has constructed an uncountable computable Polish space where all computable points are isolated, and all other points are non-hyper-arithmetic \cite{gregoriades3}.}. By Proposition \ref{prop:neqxsubspace} this yields $\neqX{\mathbf{X}} \leqW^* \neqX{\Cantor}$. Every Polish space has a total Baire space representation up to some oracle (e.g.~\cite[Corollary 4.4.12]{brattkaphd}), which via Proposition \ref{prop:neqxsurjection} shows $\neqX{\Baire} \leqW^* \neqX{\mathbf{X}}$. As $\Cantor$ and $\Baire$ are mutually embeddable, Proposition \ref{prop:neqxsubspace} also implies that $\neqX{\Baire} \equivW \neqX{\Cantor}$.
\end{proof}
\end{corollary}

\begin{corollary}
$\neqX{\mathcal{O}(\mathbb{N})} \equivsW \neqX{\Baire}$
\begin{proof}
On the one hand $\Baire$ embeds into $\mathcal{O}(\mathbb{N})$ as a subspace, on the other hand $\mathcal{O}(\mathbb{N})$ has a representation with domain $\Baire$. Thus the claim follows from Propositions \ref{prop:neqxsubspace}, \ref{prop:neqxsurjection}.
\end{proof}
\end{corollary}

For Sierpi\'{n}ski space $\mathbb{S}$ we find that the least degree of discontinuity of a multivalued function with codomain $\mathbb{S}$ is actually no less than the least degree of a discontinuous \emph{function} with codomain $\mathbb{S}$:

\begin{proposition}
\label{prop:neqS}
$\neqX{\mathbb{S}} \equivsW (\mathalpha{\neg} : \mathbb{S} \to \mathbb{S}) \equivW \lpo$
\end{proposition}
\begin{proof}
Notice that $\mathbb{S}$ is multiretraceable, so the first (strong) equivalence follows from Proposition \ref{prop:eq-dom}. Moreover, it is obvious that $(\mathalpha{\neg}:\mathbb{S}\to\mathbb{S})\leqsW \lpo$. Whereas it is easy to see that $\lpo\leqW(\mathalpha{\neg}:\mathbb{S}\to\mathbb{S})$, a straightforward continuity argument yields that this reduction cannot be made strong.
\end{proof}

\begin{corollary}
Let $P : \mathbf{X} \mto \mathbb{S}$ be discontinuous. Then $\lpo \leqW^* P$.
\begin{proof}
By combining Theorem \ref{theo:neqminimal} with Proposition \ref{prop:neqS}, together with the straight-forward observation that generalized Wadge games for multivalued functions with codomain $\mathbb{S}$ are always determined.
\end{proof}
\end{corollary}

The following shows that two scattered spaces of Cantor-Bendixson rank $2$ can yield different least degrees of discontinuity:
\begin{proposition}
$\neqX{\mathbb{N} \times (\omone)} \leW \neqX{\frac{\Nb}{\Nb}}$
\begin{proof}
The reduction follows via Proposition \ref{prop:neqxsurjection}. To see that it is strict we assume for the sake of a contradiction that $\neqX{\frac{\Nb}{\Nb}} \leqW^* \neqX{\mathbb{N} \times (\omone)}$ via $K$ and $H$. We can readily verify that the inner reduction witness $H$ needs to map the completely undefined name in the completion of $\frac{\Nb}{\Nb}$ to the completely undefined name in the completion of $\mathbb{N} \times (\omone)$. In particular, $K$ will start outputting a name for some $\frac{\cdot}{f(n)}$ upon having read sufficiently long prefixes of the completely undefined name and of $(n, \omega) \in \mathbb{N} \times (\omone)$. Now consider what happens if the original input is changing from being undefined to being $\frac{\cdot}{f(n)}$ after that prefix has passed. For the reduction to work, it needs to be the case that $H$ is mapping this name for $\frac{\cdot}{f(n)}$ to $(n,\omega)$, and that $K$ adjusts its output to some $\frac{g(m)}{\cdot}$ depending on which $(n, m)$ it receives from the oracle. However, we can again let a sufficiently long prefix pass, and now adjust the original input to be a name for $\frac{g(m)}{\cdot}$ for some $m$ such that $H$ has already forgone the opportunity to output $(n,m)$ given the prefix already read. This constitutes the desired contradiction.
\end{proof}
\end{proposition}

\subsection{More dichotomies in the continuous first-order Weihrauch degrees}
\label{subsec:dichotomies}

As shown in \cite{pauly-nobrega}, every lower cone in the Weihrauch degrees can be characterized by a Wadge-style game. We can thus use the core idea of Theorem \ref{theo:neqminimal} to identify further dichotomies. In principle, we can obtain a dual to any given Weihrauch degree (assuming enough determinacy) (i.e.~pairs $f$, $g$ such that any Weihrauch degree $h$ satisfies exactly one of $f \leqW^* h$ and $h \leqW^* g$). However, in general the construction will be very cumbersome and come with rather strong determinacy requirements. If we restrict to first-order problems, we can actually obtain some insight, as we shall demonstrate next.

Our starting point is a higher-level counterpart to $\accn$:

\begin{definition}
The problem $\Pi^0_2\mathrm{ACC}_\mathbb{N}$ receives as input a $\Pi^0_2$-subset $A \subseteq \mathbb{N}$ with $|\mathbb{N} \setminus A| \leq 1$ and returns some $i \in A$.
\end{definition}

It is often more convenient to think about $\Pi^0_2\mathrm{ACC}_\mathbb{N}$ as the problem ``Given some $p \in \Baire$, find some $k \in \mathbb{N}$ such that $\lim_{n \to \infty} p(n) \neq k$'', where the inequality is assumed to be satisfied if the limit does not exist. Our interest in $\Pi^0_2\mathrm{ACC}_\mathbb{N}$ stems from the fact that for first-order problems, it forms a dichotomy with $\C_\mathbb{N}$ as the following two theorems (which share a proof) show:

\begin{theorem}[$\mathrm{ZF} + \mathrm{DC} + \mathrm{AD}$]
For any problem $f$ exactly one of the following holds:
\begin{enumerate}
\item $^1(f) \leqW^* \C_\mathbb{N}$
\item $\Pi^0_2\mathrm{ACC}_\mathbb{N} \leqW^* f$
\end{enumerate}
\end{theorem}

\begin{theorem}[$\mathrm{ZFC}$]
Let $f : \subseteq \Baire \mto \mathbb{N}$ be such that each $f^{-1}(n)$ is Borel (in $\Baire$). Then exactly one of the following holds:
\begin{enumerate}
\item $f \leqW^* \C_\mathbb{N}$
\item $\Pi^0_2\mathrm{ACC}_\mathbb{N} \leqW^* f$
\end{enumerate}
\end{theorem}

\begin{proof}
Recall that, as shown in \cite[Theorem 7.11]{unif-lbt-brattka-de-brecht-pauly}, a multivalued function $f$ reduces to $\C_{\mathbb{N}}$ if and only if it is computable with finitely many mind-changes. As noted in \cite{pauly-nobrega}, these multfunctions admit a nice characterization in terms of games, as we now explain. We call the \emph{backtrack game for $f$ }(whose definition is implicit in \cite{wesep}) the following: player $\mathrm{I}$ plays a natural number at every turn, whereas player $\mathrm{II}$ has three options, namely she either plays a natural number, or she skips, or she erases all of her previous moves, with the proviso that she can only use this third option finitely many times. Player $\mathrm{II}$ wins if either Player $\mathrm{I}$ fails to build a point $x$ in the domain of $f$, or if she builds a point in $f(x)$ following the rules. It is clear that $f$ is computable with finitely many mind changes if and only if Player $\mathrm{II}$ has a winning strategy in the backtrack game for $f$.

Suppose now that $f$ has codomain $\mathbb{N}$, and consider the following game: at every turn, both Player $\mathrm{I}$ and Player $\mathrm{II}$ play a natural number, and 
\begin{enumerate}
\item The numbers played by Player $\mathrm{I}$ need to form some $p \in \dom(f)$, otherwise he loses.
\item From some time onwards, Player $\mathrm{II}$ needs to play a constant number $n$, otherwise she loses.
\item Finally, Player $\mathrm{II}$ wins iff her ultimately choice $n$ satisfies $n \in f(p)$.
\end{enumerate}
It is not difficult to see that Player $\mathrm{II}$ has a winning strategy in the game above if and only if it has a winning strategy in the backtrack game for $f$. Let us see the right-to-left implication, the other one being analogous: suppose Player $\mathrm{I}$ is playing according to strategy $\lambda$, and let let $p$ be the infinite string such that, for every $i$, $p(i)$ is equal to the last number played by Player $\mathrm{II}$ before turn $i$ when she plays according to $\lambda$. This way, we forget about all her mind-changes and skips, an only keep the moves in which she played a number. Since playing in $\mathbb{N}$ means that she has only one meaningful move, at some point the sequence $p$ stabilizes. It is clear that if $\lambda$ was winning and Player $\mathrm{I}$ produced the point $x$, then $\lim p\in f(x)$.

Suppose now that $f$ does not reduce to $\C_\mathbb{N}$. By either AD or Borel determinacy, we can suppose that Player \textit{I} has a winning strategy in the game described above.

Thanks to the formulation given above of $\Pi^0_2\mathrm{ACC}_\mathbb{N}$, we see that winning strategy $\lambda$ of Player $\mathrm{I}$ for this game yields a witness for the reduction $\Pi^0_2\mathrm{ACC}_\mathbb{N} \leqW^* f$: given $p\in \dom \Pi^0_2\mathrm{ACC}_\mathbb{N}$, we produce a $g(p)\in \dom f$ exactly as we did in Theorem \ref{theo:neqminimal}, namely by putting $g(p)=\lambda(p_{<0})\lambda(p_{<1})...$. The other half of the reduction is given by the identity functional. 

Finally, it is enough to observe that $\Pi^0_2\mathrm{ACC}_\mathbb{N} \not\leqW^* \C_{\mathbb{N}}$ to obtain both theorems.
\end{proof}

\section{Characterising admissibility via discontinuity of functions}

As a side note, we observe that the minimal degree of discontinuity for certain functions into a space can characterize the admissibly represented spaces (amongst the $T_0$-spaces). Admissibility of represented spaces fundamentally characterizes which represented spaces can be treated as topological spaces. More formally, the categories of admissibly represented spaces and of $\mathrm{QCB}_0$ topological spaces coincide, and are full subcategories of the categories of represented spaces and of sequential topological spaces respectively. Much of the investigation of this notion is due to Schr\"oder \cite{schroder,schroder5}.

\begin{proposition}
The following are equivalent for a $T_0$ represented space $\mathbf{X}$:
\begin{enumerate}
\item $\mathbf{X}$ is admissible.
\item Every function $f : (\omone) \to \mathbf{X}$ is either continuous (in the represented space sense) or satisfies $\lpo \equivW^* f$
\end{enumerate}
\begin{proof}
$1 \Rightarrow 2$: Since $\mathbf{X}$ is admissible, for functions into $\mathbf{X}$ having a continuous realizer coincides with topological continuity. For a function of type $f : (\omone) \to \mathbf{X}$ topological continuity can only fail there is some $U \in \mathcal{O}(\mathbf{X})$ such that $\omega \in f^{-1}(U)$, but for infinitely many $n \in \omega$ it holds that
$n \notin f^{-1}(U)$. Consider the continuous multivalued function $H : \mathbb{S} \to (\omone)$ that maps $\top$ to any $n \in \omega$ with $n \notin f^{-1}(U)$ and $\bot$ to $\omega$; and the continuous characteristic function $\chi_U : \mathbf{X} \to \mathbb{S}$. We find that
$\chi_U \circ f \circ H = (\neg : \mathbb{S} \to \mathbb{S})$, hence $\lpo \leqW^* f$. The reduction $f \leqW^* \lpo$ holds for any function
f with domain $\omone$.

$2 \Rightarrow 1$: We have already established that the dichotomy in (2) ensures that for functions of
type $f : (\omone) \to \mathbf{X}$, represented space continuity and realizer continuity coincide. We then note
that this property suffices in place of admissibility in the proof of \cite[Lemma 11]{schroder} (establishing
that admissibly represented spaces have countable pseudobasis). We conclude that $\mathbf{X}$ has a
countable pseudobase. By \cite[Theorem 12]{schroder}, as $\mathbf{X}$ is also $T_0$, it is admissible.
\end{proof}
\end{proposition}

To see that we cannot drop the $T_0$ requirement from the preceding proposition, consider
the space $\mathfrak{T}$ of Turing degrees. This is a non-trivial space with the indiscrete topology, so in
particular not admissible. However, every function $f : (\omone) \to \mathfrak{T}$ is continuous.

\begin{example}
Consider the decimal reals $\mathbb{R}_{10}$. As this space is $T_0$ but not admissible, we know
that there must be a function $f : \omone \to\mathbb{R}_{10}$ with $1 \leW^* f \leW^* \lpo$. We find one such example
in $f(n) = (-1)^n2^{
-n}$, $f(\omega) = 0$.  This function satisfies $f \equivW \C_2$.
\end{example}

\section{Outlook}
We believe that our main result, Theorem \ref{theo:main} is the $n = 1$ case of a more general result, which we hope to provide in future work. There is further potential follow-up research. The possible directions include:

\begin{enumerate}
\item How does the Weihrauch degree of $\neqX{\mathbf{X}}$ vary with properties of the space $\mathbf{X}$? Of particular interest would be to classify $\neqX{\mathbb{Q}}$ (where $\mathbb{Q}$ is equipped with the Euclidean topology); compare the discussion regarding the degree of overt choice on $\mathbb{Q}$ in \cite{overt-choice-debrecht-pauly-schroder}.
\item What are further interesting dichotomies in the continuous Weihrauch degrees (or substructures thereof), akin to the results in Subsection \ref{subsec:dichotomies}?
\item Can we fully classify the continuous Weihrauch degrees of multivalued functions with domain $(\omone)$? Besides our result in Section \ref{sec:domainanalysis}, also the notion of weak $k$-continuity and results thereabout from \cite{uftring2} seem relevant.
\end{enumerate}

\section*{Acknowledgements}

The authors are very grateful to Vittorio Cipriani, Eike Neumann, Cécilia Pradic and Manlio Valenti for many helpful discussions on several parts of the paper. Soldà's work was funded by an LMS Early Career Fellowship (ref. 2021-20), and by the FWO grant G0F8421N.

\bibliographystyle{eptcs}
\bibliography{References}

\end{document}